\documentclass[11pt]{article}
\usepackage[colorlinks=false]{hyperref}
\usepackage{amsfonts, amsmath, amssymb, amsthm}
\usepackage[T1]{fontenc} 
\usepackage{enumitem, tikz}
\definecolor{darkgray}{RGB}{64,64,64}
\definecolor{litegray}{RGB}{192,192,192}
\usetikzlibrary{quotes, angles} 
\usepackage{subcaption}

\author{Zilin Jiang\thanks{Department of Mathematics, Technion -- Israel
Institute of Technology, Technion City, Haifa 32000.}${\
}^{,}$\footnote{Email: {\tt
jiangzilin@technion.ac.il}. Supported in part by ISF grant nos 1162/15, 936/16.} \and Alexandr Polyanskii\footnotemark[1]${\
}^{,}$\footnote{Moscow Institute of Physics and Technology and Institute for
Information Transmission Problems RAS. Email: {\tt
alexander.polyanskii@yandex.ru}. Supported in part by ISF grant no. 409/16, and
by the Russian Foundation for Basic Research through grant nos 15-01-99563 A,
15-01-03530 A.}}
\title{Proof of L\'aszl\'o Fejes T\'oth's zone conjecture}
\date{}

\newtheorem{theorem}{Theorem}
\newtheorem{corollary}[theorem]{Corollary}
\newtheorem{conjecture}{Conjecture}
\newtheorem{lemma}[theorem]{Lemma}
\newtheorem{remark}{Remark}

\newcommand{\abs}[1]{\left\lvert {#1}\right\rvert}
\newcommand{\dset}[2]{\left\{#1 : #2\right\}}
\newcommand{\sset}[1]{\left\{#1\right\}}

\newcommand{\al}{\alpha}
\newcommand{\w}{\mathbf{w}}
\newcommand{\uu}{\mathbf{u}}

\newcommand{\R}{\mathbb{R}}
\renewcommand{\epsilon}{\varepsilon}

\begin{document}

\maketitle

\begin{abstract}
  A zone of width $\omega$ on the unit sphere is the set of points within
  spherical distance $\omega/2$ of a given great circle. We show that the total
  width of any collection of zones covering the unit sphere is at least $\pi$,
  answering a question of Fejes T\'oth from 1973.
\end{abstract}

\section{Introduction}\label{intro}

A \emph{plank} (or \emph{slab}, or \emph{strip}) of width $w$ is a part of the
$d$-dimensional Euclidean space $\R^d$ that lies between two parallel
hyperplanes at distance $w$. Given a convex body $C$, its \emph{width} is the
smallest $w$ such that a plank of width $w$ covers $C$. In 1932, in the context
of ``degree of equivalence of polygons'', Tarski wrote in
\cite{tarski1932uwagi}\footnote{See \cite[Chapter 7.3, 7.4]{MR3307383} for
English translations of \cite{moese1932} and \cite{tarski1932uwagi}.}, and we
quote from its English translation \cite[Chapter 7.4]{MR3307383},

\begin{quote}
  The width of the narrowest strip covering a plane figure $F$ we shall call
  the \emph{width of figure} $F$, and we shall denote it by the symbol
  $\omega(F)$. ...
  
  \textbf{A.} \emph{If a figure $F$ contains in itself, as a part, a disk with
  diameter equal to the width of the figure (for example, if figure $F$ is a
  disk or a parallelogram) and if, moreover, we subdivide this figure into any
  $n$ parts $C_1, C_2, \ldots, C_n$, then}
  \[
    \omega(F) \le \omega(C_1) + \omega(C_2) + \dots + \omega(C_n).
  \]
  \emph{Proof.} For the case in which figure $F$ is a disk, the proof is an
  almost word-for-word repetition of the proof of lemma I in the cited article
  by Moese. For the general case, ...
\end{quote}

Here ``the cited article by Moese'' refers to \cite{moese1932}\footnotemark[1].
For a long period of time, the authors of the current paper mistakenly
overlooked the contribution of Moese.

The following conjecture, which is attributed to Tarski, seems to first appear
in \cite{MR0038085}.

\begin{quote}
  \textbf{Tarski's plank problem.} If a convex body of width $w$ is covered by
  a collection of planks in $\R^d$, then the total width of the planks is at
  least $w$.
\end{quote}

It took almost twenty years before Bang proved Tarski's conjecture in his
memorable papers \cite{MR0038085, MR0046672}. At the end of his paper, Bang
asked whether his theorem could be strengthened by asking that the width of
each plank should be measured relative to the width of the convex body being
covered, in the direction normal to the plank. Although Bang's conjecture still
remains open, Ball~\cite{MR1106748} established it for centrally symmetric
convex bodies (see \cite{MR1832555} for the complex analogue). Ball also used
Bang's proof to improve the optimal density of lattice
packings~\cite{MR1191572}.

Variants of Tarski's plank problem continue to generate interest in the
geometric and analytic aspects of coverings of a convex body (see
\cite{MR3204554} for a recent survey). Most of these variants find its
generalization in spherical geometry. Bezdek and Schneider~\cite{MR2650488}
showed that the total inradius of the spherical convex domains%
\footnote{A \emph{convex spherical domain} is a domain any two points of which
can be joined by an arc of a great circle lying in the domain. The inradius of
the domain $C$ is the spherical radius of the largest cap contained in $C$.}
covering the $n$-dimensional unit sphere is at least $\pi$ (see~\cite[Section
6]{MR2957644} for a simplified proof by Akopyan and Karasev, and see
\cite{MR2111950} for the Euclidean version by Kadets).

In this paper, we are concerned with a spherical analogue of a plank: a
\emph{zone} of width $\omega$ on the 2-dimensional unit sphere is defined as
the set of points within spherical distance $\omega/2$ of a given great circle.
In 1973, Fejes T\'oth \cite{MR1537207} conjectured that if $n$ equal zones
cover the sphere then their width is at least $\omega_n = \pi / n$. Rosta
\cite{MR0345005} and Linhart \cite{MR0358576} proved the special case of 3 and
4 zones respectively. Lower bounds for $\omega_n$ were established by Fodor,
V\'igh and Zarn\'ocz \cite{MR3518649}.

Fejes T\'oth also formulated the generalized conjecture, which has been
reiterated in \cite[Chapter 3.4, Conjecture 5]{MR2163782} and the arXiv version
of \cite[Conjecture 8.3]{MR2957644}, for any set of zones (not necessarily of
the same width) covering the unit sphere.

\begin{quote}
  \textbf{Fejes T\'oth's zone conjecture.} The total width of any set of zones
  covering the sphere is at least $\pi$.
\end{quote}

We completely resolve this conjecture and generalize it for the $d$-dimensional
unit sphere $S^d$. Hereafter, all $d$-spheres are embedded in $\R^{d+1}$ and
centered at the origin. For us, a \emph{great sphere} is the intersection of a
$d$-sphere and a hyperplane passing through the origin. This extends the notion
of a great circle on a 2-dimensional sphere. Given a great sphere $C$, a
\emph{zone} $P$ of width $\omega$ on $S^d$ is defined as the set of points
within spherical distance $\omega/2$ of $C$, and the \emph{central hyperplane}
of $P$ is the hyperplane through $C$.

\begin{theorem}\label{main1}
  The total width of any collection of zones $P_1, \dots, P_n$ covering the
  unit $d$-sphere is at least $\pi$. For all $i\in[n]$, let $2\al_i$ be the
  width of $P_i$, and let $\ell_i$ be the line through the origin perpendicular
  to the central hyperplane of $P_i$. Equality holds if and only if after
  reordering the zones, $\ell_1, \dots, \ell_n$ are coplanar lines in clockwise
  order such that the angle between $\ell_i$ and $\ell_{i+1}$ equals to $\al_i
  + \al_{i+1}$ for all $i\in [n]$ under the convention that $\ell_{n+1} =
  \ell_1$ and $\al_{n+1} = \al_1$.
\end{theorem}

\begin{remark}
  Note that this result does not follow from \cite{MR2650488} because the zones 
  are not convex in spherical geometry.
\end{remark}

Our result is also connected to covering of a non-separable family of
homothetic copies of a convex body. The first result in this direction was due
to Goodman and Goodman \cite{MR0013513}. They proved the following theorem
conjectured by Erd\H{o}s.

\begin{theorem}\label{2goodman}
  Let the circles with radii $r_1, \dots, r_n$ lie in a plane. If no line
  divides the circles into two non-empty sets without intersecting any
  circle, then the given circles can be covered by a circle of radius $r =
  \sum_{i=1}^n{r_i}$.
\end{theorem}

Its potential spherical relative should concern caps: a \emph{cap} of spherical
radius $r$ on the unit $d$-sphere is defined as the set of points within
spherical distance $r$ of a given point on the sphere. Recall that the
projective duality in $\R^{d+1}$ interchanges lines through the origin with
hyperplanes through the origin. By restricting to the $d$-sphere, we associate
a great sphere to every point on the sphere. Thus a zone can be viewed as the
dual of a cap. A result in reminiscence of Theorem~\ref{2goodman} follows
immediately from Theorem~\ref{main1} and projective duality.

\begin{corollary} \label{main_dual}
  Let the caps with spherical radii $r_1, \dots, r_n$ lie on the unit
  $d$-sphere. If every great sphere intersects at least one cap, then
  the total spherical radius $\sum_{i=1}^n r_i$ is at least $\pi /2$.
\end{corollary}

The proof of Theorem~\ref{main1} will be given in the next section. In
Section~\ref{corollaries}, we state a few more corollaries, one of which
answers in the affirmative the following question of Fejes T\'oth \cite{MR1537207} for the case of centrally symmetric domains:
\begin{quote}
  \textbf{Fejes T\'oth's zone problem.} Is it true that if a convex spherical
  domain is covered with a set of zones of total width $w$, then it can be
  covered with one zone of width $w$?
\end{quote}
We conclude, in Section~\ref{open}, with some brief remarks.

\section{Proof of Theorem~\ref{main1}}

Our proof is inspired by Bang \cite{MR0046672}, Goodman and Goodman
\cite{MR0013513}. Roughly speaking, the idea of the proof is to show that
\begin{enumerate}[noitemsep]
  \item Either a particular discrete subset not fully covered by the convex
  hulls (taken in $\R^{d+1}$) of the zones is contained in the unit ball;
  \item Or several zones can be covered and replaced by one zone without
  increasing the total width.
\end{enumerate}

In this section, we denote points by vectors leading to them from the
origin. We need the following technical lemmas, the first of which is in the
spirit of Theorem~\ref{2goodman}.

\begin{lemma}\label{replacement}
  Suppose the total spherical radius of a collection of caps $D_1, \dots, D_n$
  on $S^d$ is at most $\pi/2$. For all $i\in[n]$, let $\al_i$ be the spherical
  radius of $D_i$, let $\uu_i$ be the center of $D_i$, and set $\w_i :=
  \sin\al_i\uu_i$. Denote $\w := \sum_{i=1}^n \w_i$ and $\al :=
  \sum_{i=1}^n\al_i$. If
  \begin{equation}\label{replacement_cond1}
    \abs{\w} \ge \sin\al \text{ and } \abs{\w-\w_i} \le
    \sin\left(\al-\al_i\right)\text{ for all }i\in[n],
  \end{equation}
  then a cap of spherical radius $\le \al$ can cover $\cup_{i=1}^n D_i$,
  and moreover, a cap of spherical radius $< \al$ can cover $\cup_{i=1}^n D_i$
  except when
  \begin{equation}\label{replacement_cond2}
    \abs{\w} = \sin\al \text{ and } \abs{\w-\w_i} =
    \sin\left(\al-\al_i\right) \text{ for some }i\in[n].
  \end{equation}
\end{lemma}

\begin{remark}
  In case that \eqref{replacement_cond2} holds, by the law of sines,
  $\angle(\w, \w_i) = \al - \al_i$ (see Figure~\ref{triangle}).
\end{remark}

\begin{proof}
  We claim that the cap $D$ centered at $\uu := \w / \abs{\w}$ of spherical
  radius $\al$ covers $\cup_{i=1}^n D_i$. Suppose on the contrary that say
  $D_i$ is not covered by $D$. Clearly, the angle between $\uu$ and $\uu_i$ is
  strictly greater than $\al' := \al - \al_i$. By the law of cosines and the
  monotonicity of cosine on $(0, \pi)$, we get
  \begin{equation}\label{cosine}
    \begin{split}
      \abs{\w - \w_i}^2 = \abs{\w}^2 + \abs{\w_i}^2 -
      2\abs{\w}\abs{\w_i}\cos\angle(\uu, \uu_i) > \abs{\w}^2 + \sin^2{\al_i} -
      2\abs{\w}\sin{\al_i}\cos\al' \\
      = \left(\abs{\w}-\sin\al_i\cos\al'\right)^2 +
      \sin^2\al_i\sin^2\al'.
    \end{split}
  \end{equation}
  Notice that $\abs{\w} \ge \sin\al = \sin\al_i\cos\al' + \cos\al_i\sin\al' >
  \sin\al_i\cos\al'$. The right hand side of \eqref{cosine} is at least
  \[
    \left(\sin\al - \sin\al_i\cos\al'\right)^2 + \sin^2\al_i\sin^2\al' =
    \cos^2\al_i\sin^2\al' + \sin^2\al_i\sin^2\al' = \sin^2\al'.
  \]
  Therefore $\abs{\w-\w_i} > \sin\al' = \sin(\al-\al_i)$ contradicts
  \eqref{replacement_cond1}.
  
  Clearly if $\w > \sin\al$ or $\abs{\w-\w_i} < \sin(\al - \al_i)$, then the
  cap centered at $\uu$ of spherical radius $\al - \epsilon_i$ covers $D_i$ for
  some $\epsilon_i > 0$. Hence if $\abs{\w} > \sin\al$ or $\abs{\w-\w_i} <
  \sin(\al-\al_i)$ for all $i\in[n]$, then the cap centered at $\uu$ of
  spherical radius $\al-\min\dset{\epsilon_i}{i\in[n]}$ covers $\cup_{i=1}^n
  D_i$.
\end{proof}

\begin{lemma}\label{3caps}
  Suppose the total spherical radius of 3 caps $D_1, D_2, D_3$ is $\al \le
  \pi/2$. For all $i\in[3]$, let $\al_i$ be the spherical radius of $D_i$, let
  $\uu_i$ be the center of $D_i$, and set $\w_i := \sin\al_i\uu_i$. Suppose that
  \[
    \abs{\w_1 + \w_2} = \sin(\al_1 + \al_2) \text{ and }\abs{\w_1 + \w_2 + \w_3}
    = \sin\al.
  \]
  If $\w_1, \w_2, \w_3$ are not coplanar vectors, then a cap of spherical 
  radius $< \al$ can cover $D_1 \cup D_2 \cup D_3$.
\end{lemma}

\begin{proof}
  Denote $\w_{12} := \w_1 + \w_2$ and $\w_{123} := \w_1 + \w_2 + \w_3$. By the
  law of cosines and trigonometric identities, we obtain that
  \[
    \cos\angle(\w_{12}, \w_1)
    = \frac{\abs{\w_{12}}^2+\abs{\w_1}^2-\abs{\w_2}^2}{2\abs{\w_{12}}\abs{\w_1}}
    = \frac{\sin^2(\al_1+\al_2)+\sin^2\al_1-\sin^2\al_2}{2\sin(\al_1+\al_2)\sin\al_1}
    =\cos\al_2.
  \]
  Therefore $\angle(\w_{12}, \w_1) = \al_2$. Similarly, one can compute
$$\angle(\w_{12}, \w_2) = \al_1,\ \angle(\w_{123}, \w_{12}) = \al_3,\
\angle(\w_{123},\ \w_{3}) = \al_1 + \al_2.$$
  
  Suppose that $\w_1, \w_2, \w_3$ are not coplanar. Set $\w = \w(\epsilon) :=
  \w_{123} + \epsilon\w_3$. We claim that the cap $D = D(\epsilon)$ centered at
  $\w / \abs{\w}$ of spherical radius $<\al$ can cover $D_1 \cup D_2 \cup D_3$
  for some $\epsilon > 0$. Because $\w_1, \w_{12}, \w_{123}$ are not coplanar,
  by the spherical triangle inequality, we obtain
  \[
    \delta_1 := \angle(\w_{123}, \w_{12}) + \angle(\w_{12}, \w_1) -
    \angle(\w_{123}, \w_1) = \al_2 + \al_3 - \angle(\w_{123}, \w_1) > 0.
  \]
  Therefore there exists $c_1 \in \R$ such that
  \begin{equation}\label{w1}
    \angle(\w, \w_1) = \angle(\w_{123}+\epsilon\w_3, \w_1) = \angle(\w_{123},
    \w_1) + c_1\epsilon + O(\epsilon^2) = \al_2 + \al_3 - \delta_1 + c_1\epsilon
    + O(\epsilon^2).
  \end{equation}
  Similarly, $\delta_2 := \al_1 + \al_3 - \angle(\w_{123}, \w_2) > 0$ and there
  exists $c_2 \in \R$ such that
  \begin{equation}\label{w2}
    \angle(\w, \w_2) = \al_1 + \al_3 - \delta_2 + c_2\epsilon + O(\epsilon^2).
  \end{equation}
  Since $\w_{123}, \w_{12}, \w_3$ are coplanar, but not collinear, there exists
  $c_3 > 0$ such that
  \begin{equation}\label{w3}
    \angle(\w, \w_3) = \angle(\w_{123} + \epsilon\w_3, \w_3) =
    \angle(\w_{123}, \w_3) - c_3\epsilon + O(\epsilon^2) = \al_1 + \al_2 -
    c_3\epsilon + O(\epsilon^2),
  \end{equation}
  In (\ref{w1}, \ref{w2}, \ref{w3}), one can take sufficiently small $\epsilon
  > 0$ so that $\angle(\w, \w_i) < \al - \al_i$ for all $i\in [3]$. This proves
  the claim.
\end{proof}

\begin{remark}
  Because $\angle(\w_{123}, \w_1) < \al_2 + \al_3, \angle(\w_{123}, \w_2) <
  \al_1 + \al_3$ and $\angle(\w_{123}, \w_3) = \al_1 + \al_2$, the cap, say
  $D_{123}$, centered at $\w_{123}/\abs{\w_{123}}$ of spherical radius $\al =
  \al_1 + \al_2 + \al_3$ covers $\cup_{i=1}^3D_i$. The idea of the proof of
  Lemma~\ref{3caps} is to move the center of $D_{123}$ towards
  $\w_3/\abs{\w_3}$ slightly and shrink its spherical radius properly so that
  the cap would still cover $\cup_{i=1}^3D_i$ (see Figure~\ref{shrinking}).
\end{remark}

\begin{figure}
  \centering
  \begin{minipage}{.35\textwidth}
    \centering
    \begin{tikzpicture}[scale=1]
      \coordinate (O) at (0,0);
      \coordinate (A) at (1,1);
      \coordinate (B) at (0,3);
      \draw[darkgray] (O) -- (A) -- (B) -- (O)
        pic["", draw, angle eccentricity = 2, angle radius = 0.3cm]
          {angle=A--O--B}
        pic["", draw, angle eccentricity = 1.5, angle radius = 0.25cm]
          {angle=B--A--O}
        pic["$\al_i$", draw, angle eccentricity = 2.5, angle radius = 0.4cm]
          {angle=O--B--A};
      \node[anchor=north] at (0, 0) {$\al-\al_i$};
      \node[anchor=west] at (1, 1) {$\pi-\al$};
      \fill[darkgray] (0.08,0.19) circle (0.03cm);
      \fill[darkgray] (0.86,1.02) circle (0.03cm);
      \node[anchor=east] at (-0.1, 1.5) {$\abs{\w} = \sin\al$};
      \node[anchor=west] at (0.6, 2.25) {$\abs{\w - \w_i}$};
      \node[anchor=west] at (1, 1.7) {$=\sin(\al-\al_i)$};
      \node[anchor=west] at (0.6, 0.25) {$\abs{\w_i} = \sin\al_i$};
      \node at (0, 3.5) {};
      \node at (0, -0.5) {};
    \end{tikzpicture}
    \captionof{figure}{}
    \label{triangle}
  \end{minipage}%
  \begin{minipage}{.33\textwidth}
    \centering
    \def\ra{0.6}
    \def\rb{0.9}
    \def\rc{0.5}
    \def\tc{85}
    \def\d{0.2}
    \begin{tikzpicture}[scale=1.06]
      \coordinate (A) at (-\rb,0);
      \coordinate (B) at (\ra,0);
      \coordinate (C) at (\tc:\ra+\rb+\rc);
      \coordinate (D) at (0,0);
      \coordinate (E) at (\tc:\rc);
      \coordinate (F) at (\tc:\rc+\d);
      \fill[opacity=0.2] (A) circle (\ra);
      \fill[opacity=0.2] (B) circle (\rb);
      \fill[opacity=0.2] (C) circle (\rc);
      \fill[opacity=0.1] (\tc:\rc+\d) circle (\ra+\rb+\rc-\d);
      \fill[darkgray] (A) circle (0.05cm);
      \fill[darkgray] (B) circle (0.05cm);
      \fill[darkgray] (C) circle (0.05cm);
      \fill[darkgray] (D) circle (0.05cm);
      \fill[darkgray] (E) circle (0.05cm);
      \fill[darkgray] (F) circle (0.05cm);
      \draw[dashed, darkgray] (D) circle (\ra+\rb);
      \draw[dashed, darkgray] (E) circle (\ra+\rb+\rc);
      \draw[dashed, darkgray] (A) -- (B);
      \draw[dashed, darkgray] (D) -- (C);
      \node at (A) [below] {$\w_1$};
      \node at (B) [below right] {$\w_2$};
      \node at (C) [above] {$\w_3$};
      \node at (D) [below] {$\w_{12}$};
      \node at (E) [left] {$\w_{123}$};
      \node at (F) [right] {$\w$};
    \end{tikzpicture}
    \captionof{figure}{}
    \label{shrinking}
  \end{minipage}%
  \begin{minipage}{.32\textwidth}
    \centering
    \begin{tikzpicture}[scale=1.06]
      \coordinate (O) at (3.5,0);
      \coordinate (A) at (0.5,1.5);
      \coordinate (B) at (0.5,-1);
      \draw[darkgray, dashed] (0,-1.25) -- (4,-1.25);
      \draw[darkgray, dashed] (0,1.25) -- (4,1.25);
      \draw[darkgray, dashed] (0,0) -- (4,0);
      \draw[darkgray] (0.5, 0.2) -- (0.7, 0.2) -- (0.7, 0);
      \fill[darkgray] (O) circle (0.05);
      \fill[darkgray] (A) circle (0.05);
      \fill[darkgray] (B) circle (0.05);
      \draw[-latex, darkgray, shorten >= 0.05] (O) -- (A);
      \draw[darkgray] (B) -- (A);
      \draw[-latex, darkgray, shorten >= 0.05] (O) -- (B);
      \node at (4,0) [above] {$H_i$};
      \node at (2, 0.75) [above right] {$\w$};
      \node at (2, -0.5) [below right] {$\w - 2\epsilon_i\w_i$};
      \node at (0,0) [above] {$2\abs{\w_i}$};
      \node at (0,2.2) {};
      \node at (0,-1.57) {};
    \end{tikzpicture}    
    \captionof{figure}{}
    \label{hyperplane}
  \end{minipage}
\end{figure}

\begin{proof}[Proof of Theorem~\ref{main1}]
  Suppose that the total width of a collection of zones $P_1, \dots, P_n$ on
  $S^d$ is less than $\pi$. Let $2\al_i$ be the width of $P_i$, let $H_i$ be the
  central hyperplane of $P_i$, and let $\uu_i$ be a unit normal vector of $H_i$.
  Set $\w_i := \sin\al_i \uu_i$ for all $i\in[n]$. Consider the set
  \[
    L := \dset{\sum_{i=1}^n\epsilon_i\w_i}{(\epsilon_1, \dots, \epsilon_n) \in
    \sset{\pm 1}^n}.
  \]
  Let $\w := \sum_{i=1}^n \epsilon_i\w_i$ achieve the maximal norm among points
  of $L$. For every $i\in [n]$, using the fact that $\abs{\w} \ge \abs{\w -
  2\epsilon_i\w_i}$, one can easily see that the distance from $\w$ to $H_i$ is
  at least $\abs{\w_i} = \sin\al_i$ (see Figure~\ref{hyperplane}).
  
  If $\abs{\w} < 1$, then $\w / \abs{\w} \notin \cup_{i=1}^n P_i$. Otherwise,
  $\abs{\w} \ge 1$. Without loss of generality, we may assume that $\epsilon_1
  = \dots = \epsilon_n = 1$. In particular, $\abs{\sum_{i=1}^n \w_i} >
  \sin\left(\sum_{i=1}^n\al_i\right)$ because the total width $2\al_1 + \dots +
  2\al_n < \pi$. Find the minimal set $I \subset [n]$ such that
  $\abs{\sum_{i\in I}\w_i} > \sin\left(\sum_{i\in I}\al_i\right)$. For each
  $i\in I$, denote by $D_i$ the cap centered at $\uu_i$ of spherical radius
  $\al_i$. By the minimality of $I$, we can apply the first part of
  Lemma~\ref{replacement} to $\dset{D_i}{i\in I}$ to find a cap, say $D$,
  covering $\cup_{i\in I}D_i$ of spherical radius $\le \sum_{i\in I}\al_i$.
  Using the projective duality, one can check that the dual of $D_i$ is $P_i$
  and the dual of $D$ is a zone, say $P$, of width $\le \sum_{i\in I}\al_i$.
  Moreover, because $\cup_{i\in I} D_i$ is covered by $D$ and the projective
  duality preserves containment, $\cup_{i\in I} P_i$ is covered by $P$.
  Clearly, $\abs{I} > 1$. Now we can reduce the number of zones by replacing
  $\dset{P_i}{i\in I}$ by $P$ without increasing the total width, and we repeat
  the above argument for zones $\sset{P} \cup \dset{P_i}{i\in[n]\setminus I}$.
  
  Finally, we investigate the case when equality holds. Suppose that zones
  $P_1, \dots, P_n$ of total width $2\al = \pi$ cover $S^d$. For all $i \in
  [n]$, let $D_i$ be the dual cap\footnote{Strictly speaking, the dual of a
  zone $P$ is a pair of equal caps. However, if a normal vector $\uu$ of the
  central hyperplane $H$ of $P$ is chosen, then the dual cap of $P$ is defined
  as one of these caps in the half-space determined by $H$ containing $\uu$.}
  of $P_i$. From the previous argument, we must have $\abs{\w} = 1 = \sin\al$.
  Moreover, for any $I \subset [n]$, caps $\dset{D_i}{i\in I}$ cannot be
  covered by a cap of spherical radius $< \sum_{i\in I}\al_i$ since otherwise
  we can cover the sphere with zones of total width $<\pi$.
  
  The second part of Lemma~\ref{replacement} shows that $\abs{\w - \w_i} =
  \sin(\al - \al_i)$ for some $i\in[n]$. Without loss of generality, assume that
  $\abs{\w - \w_n} = \sin(\al - \al_n)$. Again, apply the second part of
  Lemma~\ref{replacement} to $D_1, \dots, D_{n-1}$, it must be the case that
  $\abs{\w - \w_n - \w_i} = \sin(\al - \al_n - \al_i)$ for some $i\in[n-1]$.
  After repeating this argument for $(n-1)$ times, we can assume, without loss
  of generality, that
  \[
    \abs{\w_1 + \dots + \w_i} = \sin(\al_1 + \dots + \al_i) \text{ for all
    }i\in[n].
  \]
  
  Since $\abs{\w_1 + \w_2} = \sin(\al_1 + \al_2)$, $\abs{\w_1 + \w_2 + \w_3} =
  \sin(\al_1 + \al_2 + \al_3)$ and $D_1, D_2, D_3$ cannot be covered by a cap
  of spherical radius $<\al_1 + \al_2 + \al_3$, Lemma~\ref{3caps} implies that
  $\w_1, \w_2, \w_3$ are coplanar. Now let $\w_2' := \w_1 + \w_2$ and $D_2'$ be
  the cap centered at $\w_2' / \abs{\w_2'}$ of spherical radius $\al_2' :=
  \al_1 + \al_2$. One can check that $D_2'$ cover $D_1, D_2$, and so $D_2',
  D_3, D_4$ cannot be covered by a cap with radius $<\al_2' + \al_3 + \al_4$.
  Since, in addition, $\abs{\w_2' + \w_3} = \sin(\al_2' + \al_3)$ and
  $\abs{\w_2' + \w_3 + \w_4} = \sin(\al_2' + \al_3 + \al_4)$, Lemma~\ref{3caps}
  implies that $\w_2', \w_3, \w_4$ are coplanar, hence $\w_1, \w_2, \w_3, \w_4$
  are coplanar. Repeating this argument, we obtain that $\w_1, \dots, \w_n$ are
  coplanar. This reduces the problem for $S^d$ to $S^1$, and it is easy to
  check the equality condition for $S^1$.
\end{proof}

\begin{remark}
  Alternatively, to find a point of $L$ that is not covered with the convex
  hulls of the zones, one could use Bang's approach~\cite{MR0046672} by
  optimizing a cleverly chosen quadratic function on $\sset{\pm 1}^n$. Our
  proof makes use of the neat trick, due to Bogn\'ar~\cite{MR0133057}, by
  picking the vector achieving the maximal norm among points of $L$.
\end{remark}

\section{Other corollaries}\label{corollaries}

The following two corollaries follow easily from Theorem~\ref{main1}. We call a
pair of equal caps \emph{an antipodal cap} if they are antipodal to each other.

\begin{corollary}
  Given a family of antipodal caps of radii $r_1, \dots, r_n$ on the unit
  $d$-sphere, it is always possible to find a point common to them if the total
  spherical radius $\sum_{i=1}^n r_i$ is at least $(n-1)\frac{\pi}{2}$.
\end{corollary}

\begin{proof}
  Apply Theorem~\ref{main1} to the complement of the antipodal caps.
\end{proof}

\begin{corollary}
  Given a family of great spheres $C_1, \dots, C_n$ on the unit $d$-sphere, it
  is always possible to find an open cap of spherical radius $\pi/2n$ not
  intersecting any $C_i$.
\end{corollary}

\begin{proof}
  Apply Theorem~\ref{main1} to the zones of width $\pi/n$ with central great
  spheres $C_1, \dots, C_n$.
\end{proof}

A spherical domain $W$ is \emph{centrally symmetric} with respect to a point
$c$ if for every $a \in W$ there is $b \in W$ such that $c$ is the middle point
of the shortest arc connecting $a$ and $b$. The following corollary partially
confirms Fejes T\'oth's zone problem.

\begin{corollary}\label{ft}
  If a centrally symmetric convex spherical domain is covered with a set of
  zones of total width $w$, then it can be covered with one zone of width $w$.
  In particular, the total width of zones covering a cap of spherical radius
  $r$ is at least $2r$.
\end{corollary}

\begin{proof}
  Suppose a convex spherical domain $W$ is centrally symmetric with respect to
  $c$, and is covered with a set of zones $P_1, \dots, P_n$ of total width $w$.
  Suppose $D$ is the largest cap contained in $W$ and centered at $c$. Let $r$
  be the spherical radius of $D$. Let $D'$ be the equal cap that is antipodal
  to $D$. The complement of $D\cup D'$ is a zone of width $\pi - 2r$, which
  together with $P_1, \dots, P_n$ covers the sphere. Theorem~\ref{main1} says
  that $\pi - 2r + w \ge \pi$ or equivalently $2r \le w$. Let $D$ touch the
  boundary of $W$ at $a, b$ which are centrally symmetric with respect to $c$.
  One can check that the zone of width $2r$ with central great circle passing
  through $c$ and perpendicular to arc $ab$ covers $W$.
\end{proof}

\section{Concluding remarks}\label{open}

Besides Fejes T\'oth's zone problem at the end of Section~\ref{intro}, let us
mention four more conjectures related to the theme of the article. We believe
the following strengthening of Theorem~\ref{main1} holds.

\begin{conjecture}
  A spherical segment is the solid defined by cutting the unit ball with a pair
  of parallel planes. Its width is the length of the shortest arc on the sphere
  whose endpoints touch both parallel planes. If the unit ball is covered with
  a collection of spherical segments, then the total width of the spherical
  segments is at least $\pi$.
\end{conjecture}

We believe that an even stronger version holds.

\begin{conjecture}
  If the unit sphere is covered with a collection of spherical segments, then
  the total width of the spherical segments is at least $\pi$.
\end{conjecture}

In connection with the spherical relative of Theorem~\ref{2goodman}, we suggest
the following conjecture.

\begin{conjecture} \label{spherical_2goodman}
  Let the caps with spherical radii $r_1, \dots, r_n$ lie on a hemisphere of
  the unit $d$-sphere. If no great sphere divides the caps into two non-empty
  sets without intersecting any cap, then the given caps can be covered by a
  cap of spherical radius $r = \sum_{i=1}^n r_i$.
\end{conjecture}

\begin{remark}
  The converse of Corollary~\ref{main_dual} is true up to rearrangement of the
  given caps in the following sense. If the total spherical radius of the caps
  is at least $\pi/2$, one can align their centers along a great sphere so that
  every great sphere intersects at least one cap. Now fix such a family of caps
  of total spherical radius $\pi/2$. Also choose an arbitrary $\pi/3$-covering
  $X$ for the $d$-sphere, that is, a finite subset of the $d$-sphere such that
  every point on the $d$-sphere is within spherical distance $\pi/3$ of some
  point in $X$. For each point in $X$, we add to the family a tiny cap centered
  at that point. Thus we obtain a non-separable family of caps of total
  spherical radius slightly larger than $\pi/2$. However, just to cover $X$, we
  need a cap of spherical radius $\ge 2\pi/3$. This construction shows the
  assumption that the given caps lie on a hemisphere is indispensable to
  Conjecture~\ref{spherical_2goodman}.
\end{remark}

Informally, the proof of Theorem~\ref{main1} takes advantage Bang's ingenious
approach~\cite{MR0046672} --- we know a bit more about the zones in case Bang's
argument does not apply. Moreover, Theorem~\ref{main1} can be seen as a
strengthening of Tarski's plank theorem~\cite{tarski1932uwagi}. The following
conjecture of Bezdek~\cite{MR2007958} shares these features.
\begin{conjecture}[Bezdek]
  Take a unit disk and turn it into an annulus by cutting a sufficiently small
  concentric circular hole. If the annulus is covered with a collection of
  planks, then the total width of the planks is at least $1$.
\end{conjecture}

We mention a conjecture attributed to Bezdek and Schneider
\cite[Problem 6.2]{MR2957644}. It was solved for $r \ge \pi / 2$
\cite[Theorem 6.1]{MR2650488}, and it is closely related to Corollary~\ref{ft}.

\begin{conjecture}[Bezdek--Schneider]
  If a cap of spherical radius $r$ is covered with a collection of convex
  spherical domains, then the total inradius of the domains is at least $r$.
\end{conjecture}

Another interesting direction of generalization is known as the fractional
plank problem (see \cite{MR1902679} and \cite[Section 7]{akopyan2014bang} for
results and conjectures). Finally, we refer the readers to \cite[Chapter
3.4]{MR2163782} for a more comprehensive list of open problems related to
Tarski's plank problem.

\section*{Acknowledgements}

We would like to thank Arseniy Akopyan and Roman Karasev for inspirational
discussions on Fejes T\'oth's zone conjecture, and for bringing the historical
account of Tarski's plank problem to our attention. We are grateful to Ron
Aharoni and Benjamin Matschke for discussions on related topics. We also
benefited from the constructive comments of the anonymous referee.

\bibliographystyle{jalpha}
\bibliography{spherical_plank}

\end{document}